\documentclass{article}

\usepackage[margin=2.5cm]{geometry}
\usepackage{enumitem, amsthm, amssymb, amsmath, amsfonts, amsthm, comment, cleveref, cjhebrew, xcolor, url}

\newcommand{\N}{\mathbb{N}}

\newcommand{\Z}{\mathbb{Z}}

\newcommand{\cA}{\mathcal{A}}
\newcommand{\cB}{\mathcal{B}}
\newcommand{\cF}{\mathcal{F}}

\newcommand{\sfX}{\mathsf{X}}

\newcommand{\eps}{\varepsilon}
\newcommand{\sm}{\setminus}
\newcommand{\sn}{\subsetneq}
\newcommand{\per}{\mathrm{per}}
\newcommand{\bl}{\mathrm{Blanks}}

\newtheorem{thm}{Theorem}
\numberwithin{thm}{section}
\newtheorem{lemma}[thm]{Lemma}
\newtheorem{prop}[thm]{Proposition}
\newtheorem{cor}[thm]{Corollary}

\theoremstyle{definition}
\newtheorem{defn}{Definition}
\numberwithin{defn}{section}

\theoremstyle{remark}
\newtheorem{rmk}[thm]{Remark}

\title{Encoding subshifts through sliding block codes}
\author{Sophie MacDonald}
\date{\today}

\begin{document}

\maketitle

\begin{abstract}
    We prove a generalization of Krieger's embedding theorem, in the spirit of zero-error information theory. Specifically, given a mixing shift of finite type $X$, a mixing sofic shift $Y$, and a surjective sliding block code $\pi: X \to Y$, we give necessary and sufficient conditions for a subshift $Z$ of topological entropy strictly lower than that of $Y$ to admit an embedding $\psi: Z \to X$ such that $\pi \circ \psi$ is injective. 
\end{abstract}

\section{Introduction}\label{intro-sec}

In a foundational paper \cite{shannon-1948-bell} in information theory, Shannon introduced a model of a noisy communication channel, in which the input and output are modeled by stationary probability measures on a space of sequences of symbols. Shannon gave conditions under which the input can be recovered from the output, at least with an acceptable rate of error or ambiguity, in the case of a Bernoulli source, and this work has since been extended to more general sources \cite{kieffer-1981-wahr}.

This paper is motivated by the particular question of when one can ensure zero error, not just almost surely as in information theory but in fact deterministically. A deterministic channel can be modeled by a sliding block code, i.e. a continuous, shift-commuting map on a subshift, on which a stationary process could be supported. In this model, we can apply techniques of symbolic dynamics to investigate the effects of \textit{deterministic noise} \cite{mpw-82-sk}, also called \textit{distortion} \cite{shannon-1948-bell}, which we can interpret as a failure of injectivity of the sliding block code representing the channel, even in the absence of random errors.

The main result of this paper, \Cref{main-thm}, determines the extent to which the non-injectivity of a sliding block code on a mixing shift of finite type (SFT) can be avoided by restricting to a subshift of the domain. Interpreting the sliding block code as a channel with deterministic noise, \Cref{main-thm} characterizes the sources with entropy strictly lower than that of the output which can be transmitted without error or ambiguity. 

\begin{thm}\label{main-thm}
Let $X$ be a mixing SFT, $Y$ a mixing sofic shift, and $\pi: X \to Y$ a factor code. Let $Z$ be a subshift with topological entropy strictly less than that of $Y$. Then there exists a subshift $Z'$ of $X$ conjugate to $Z$ such that $\pi|_{Z'}$ is injective, if and only if for every $n \geq 1$, the number of periodic points of least period $n$ in $Z$ is at most the number of periodic points of least period $n$ in $Y$ with a $\pi$-preimage of equal least period.
\end{thm}

\Cref{main-thm} is a generalization of the following theorem of Krieger in the case of unequal entropy; in particular, \Cref{main-thm} reduces to \Cref{krieger-thm} in the case that $Y=X$ and $\pi$ is the identity.

\begin{thm}[Theorem 2 in \cite{wk-82-etds}]\label{krieger-thm}
Let $Y$ be a mixing shift of finite type and $Z$ a subshift. Then there is a subshift $Z' \subseteq Y$ conjugate to $Z$ if and only if $Z$ and $Y$ are conjugate or the (topological) entropy of $Z$ is less than that of $Y$ and, for every $n \geq 1$, the number of periodic points of least period $n$ in $Z$ is at most the corresponding number in $Y$.
\end{thm}

We note that with $X,Y,Z,\pi$ as in the statement of \Cref{main-thm}, clearly there exists a subshift $Z'$ of $X$ conjugate to $Z$ such that $\pi|_{Z'}$ is injective if and only if there exists a sliding block code $\psi: Z \to X$ such that $\pi \circ \psi$ is injective, in which case $Z' = \psi(Z) \subset X$. To verify the ``only if'' statement in \Cref{main-thm}, suppose that there is a subshift $Z'$ of $X$ conjugate to $Z$ such that $\pi|_{Z'}$ is injective. Let $y \in \pi(Z')$ be periodic. Let $x = \pi|_{Z'}^{-1}(y)$ be the unique preimage of $y$ in $Z'$. Then the orbit of $x$ is in bijection with the orbit of $y$; otherwise, $\pi$ would fail to be injective on the orbit of $x$, which is contained in $Z'$. In particular, $x$ has finite orbit, so $x$ is periodic, moreover with $\per(x) = \per(y)$. Thus, every periodic point in $\pi(Z') \subset Y$ has a periodic preimage in $Z' \subset X$ of equal least period, which shows the necessity of the stated condition.

Both \Cref{main-thm} and \Cref{krieger-thm} give conditions for the existence of an embedding in terms of entropy and a periodic point condition. The following corollary, which we prove in \Cref{sec-conseq}, shows that the periodic point condition can be removed in exchange for a small loss of injectivity.

\begin{cor}\label{mmb-cor}
Let $X$ be a mixing SFT, $Y$ a mixing sofic shift, and $\pi: X \to Y$ a factor code. Let $Z$ be a subshift with topological entropy strictly less than that of $Y$. Then there exist a subshift $Z'$, a finite-to-one factor code $\chi: Z' \to Z$, and a sliding block code $\psi: Z' \to X$ such that $\pi \circ \chi$ is injective. Moreover, if $Z$ is mixing sofic with positive entropy (i.e. not a single fixed point), then $Z'$ can be taken to be a mixing SFT and $\chi$ can be taken to be almost invertible.
\end{cor}

The code $\chi$ is in fact injective except on points in $Z'$ whose images in $Z$ are backward-asymptotic to one of finitely many periodic points in $Z$. See \Cref{blowup} and \Cref{rmk-blowup-details}. From \Cref{mmb-cor}, we can immediately conclude the following, with $h$ denoting the topological entropy of a subshift.

\begin{cor}
Let $X$ be a mixing SFT, $Y$ a mixing sofic shift, and $\pi: X \to Y$ a factor code. For any $\eps > 0$, there exists a mixing SFT $Z \subset X$ with $h(Z) > h(Y) - \eps$ such that $\pi|_Z$ is injective.
\end{cor}

The proof of \Cref{main-thm} adapts the strategy used to prove \Cref{krieger-thm} in \cite{wk-82-etds,lm-95-intro} and related results in \cite{mb-1984-etds}. The outline of the proof is as follows. We use a marker set, as in the proof of \Cref{krieger-thm}, to break points in $Z$ into moderate blocks and long periodic blocks, separated by marker coordinates. We code these separately using certain ``data blocks'' in $Y$, some of moderate length and some long and periodic, where the long periodic data blocks come from periodic points with $\pi$-preimages of equal least period in $X$. A block in $Z$ between marker coordinates is coded to a data block in $Y$ which is shorter by an additive constant, so that there are gaps between the data blocks, filled with repetitions of a ``blank'' symbol. We then lift the data blocks from $Y$ to data blocks from $X$, then replace the blanks with a ``stamp'' block from $X$ to form a valid point in $X$. The stamp block is chosen to ensure that once the point in $X$ is coded into $Y$ by $\pi$, the locations of the stamp, and thus of the marker coordinates, can be recognized. These manipulations of markers, blanks, and stamps are presented in detail in \Cref{sec-code}, while the quantitative arguments required to construct the data blocks and stamps are given in \Cref{sec-count}.

The statement of \Cref{krieger-thm} is false for $X$ merely mixing sofic, and to date there is no known characterization of the subshifts that embed into a given mixing sofic shift, though some sufficient conditions are known \cite{mb-1984-etds,thomsen-2004-sofic}. \Cref{main-thm} sheds some light on this problem, without resolving it. Salo-Törmä have answered \cite{vs-it-22-mo} the following related question: let $Y$ be a mixing sofic shift and $Z \subset Y$ a mixing SFT. For which such $Y,Z$ do there exist a mixing SFT extension $\pi: X \to Y$ and a (mixing SFT) $Z' \subset X$ such that $\pi|_{Z'}: Z' \to Z$ is a conjugacy? However, it is unclear how the conditions given in that answer compare to those in \Cref{main-thm}, or to the results given in \cite{thomsen-2004-sofic}. As a final related question, when $Y$ is an SFT and $Z$ is conjugate to $Y$, the existence of an SFT $Z' \subset X$ conjugate to $Z$ such that $\pi|_{Z'}: Z' \to Y$ is a conjugacy, i.e. is surjective as well as injective, has been studied in \cite{hmw-2022}, continuing work from \cite{mpw-82-sk}.

\section{Conventions, definitions, and background on symbolic dynamics}\label{conv-sec}

\subsection{Subshifts and sliding block codes}\label{ss-sbc}

Let $\cA$ be a finite set with the discrete topology, which we will call an \textit{alphabet}. The set $\cA^{\Z}$ of bi-infinite sequences over $\cA$, equipped with the product topology, is called the \textit{full shift} over $\cA$, so called because the shift action $\sigma: \cA^{\Z} \to \cA^{\Z}$, given by $(\sigma x)_i = x_{i+1}$, is a homeomorphism. A closed, shift-invariant subset of the full shift is called a \textit{subshift}. The topology on $\cA^{\Z}$ is generated by \textit{cylinders}, which are sets of the form
\[
[w]_i := \{ x \in \cA^{\Z} \, | \, x_{i+j} = w_{j}, \, 0 \leq j \leq n -1   \},
\]
where $w \in \cA^n$ is a \textit{block} or \textit{word} of length $n \in \N$, and $i \in \Z$. Note that by shift-invariance, for any subshift $X \subset \cA^{\Z}$ and any block $w \in \cA^*$, we have $X \cap [w]_i \neq \emptyset$ for some $i \in \Z$ if and only if $X \cap [w]_i \neq \emptyset$ for all $i \in \Z$. 

A subshift $X \subset \cA^{\Z}$ is characterized by the set $\cB(X)$ of blocks $w \in \cA^*$ such that $X \cap [w] \neq \emptyset$, called the \textit{language} of $X$. When the intended subshift $X$ is clear, we write $[w]_i$ for $X \cap [w]_i$. We write $\cB_n(X) = \cB(X) \cap \cA^n$. We can equivalently characterize a subshift by a set of forbidden words $\cF \subset \cA^*$, writing $\sfX_{\cF} := \overline{\cA^{\Z} \sm \bigcup_{w \in \cF} \bigcup_{i \in \Z} [w]_i  }$. Note that in general $\cF \sn \cA^* \sm \cB(\sfX_{\cF})$. For a given subshift $X \subset \cA^{\Z}$, there may be several different sets of forbidden words $\cF \subset \cA^*$  such that $X = \sfX_{\cF}$. A \textit{shift of finite type} (SFT) is a subshift $X$ such that $X = \sfX_{\cF}$ for some finite set $\cF$. A \textit{$k$-step} SFT over $\cA$ is an SFT of the form $X = \sfX_{\cF}$ for some set $\cF \subset \cA^{k+1}$.

It is a theorem (the Curtis-Hedlund-Lyndon theorem) that, for subshifts $X,Y$, a function $\phi: X \to Y$ is continuous and shift-equivariant if and only if it is a \textit{sliding block code}, which means that there exist $m,n \geq 0$ and $\Phi: \cB_{m+n+1}(X) \to \cB_1(Y)$ such that for every $x \in X$ and every $i \in \Z$, $\phi(x)_i = \Phi(x_{[i-m,i+n]})$. We say that $\phi$ is a $k$-block code if $m+n+1=k$. A \textit{factor code} is a surjective sliding block code, and for a sliding block code $\phi$ defined on a subshift $X$, we say that the image $\phi(X)$ is a \textit{factor} of $X$, and that $X$, or more properly $\phi: X \to \phi(X)$, is an \textit{extension} of $\phi(X)$. A \textit{sofic shift} (from the Hebrew \<swpy>, ``sofi'', meaning ``finite'') is any factor of a shift of finite type. An injective sliding block code is called an \textit{embedding}, and a bijective sliding block code is called a \textit{conjugacy}. The properties of being sofic and of finite type are both invariant under conjugacy.

A subshift $X$ is said to be \textit{irreducible} if for all $u, w \in \cB(X)$, there exists $v \in \cB(X)$ such that $uvw \in \cB(X)$, and \textit{strongly irreducible} with gap $g \geq 1$ if, for any $u,w$, we can take always take $v \in \cB_g(X)$. Any factor of an irreducible (resp. strongly irreducible) subshift is irreducible (resp. strongly irreducible). A \textit{periodic point} in a subshift $X$ is a point $x \in X$ with $x = \sigma^n x$ for some $n \geq 1$---we say that $x$ has period $n$. The \textit{least period} $\per(x)$ of a periodic point $x$ is the least $n$ such that $\sigma^n x = x$. Note that $| \{ \sigma^n x \, | \, n \in \Z  \} | = \per(x)$. We write $P(X)$ for the set of periodic points in a subshift $X$, $Q_n(X)$ for the set of periodic points of least period $n$, and $q_n(X) = |Q_n(X)|$. The number of periodic points of a given least period is a conjugacy invariant.

It is a theorem that periodic points are dense in any irreducible shift of finite type. The \textit{period} $\per(X)$ of an irreducible shift of finite type $X$ is the gcd of the periods of the periodic points of $X$. An irreducible SFT with period $1$ is said to be \textit{aperiodic}. An irreducible SFT is strongly irreducible if and only if it is aperiodic, if and only if has periodic points of all sufficiently high periods. For irreducible sofic shifts, strong irreducibility is equivalent to having periodic points of all sufficiently high periods, which clearly implies that the periods have gcd $1$, but the reverse implication fails. For example, consider the odd shift over $\{ 0,1 \}$, in which the block $10^n1$ is permitted only for odd $n$. This is an irreducible sofic shift which contains the fixed point $0^{\infty}$, so the periods of periodic points trivially have gcd $1$. However, the odd shift has no other periodic points of odd period. We follow the convention of the literature in referring to strongly irreducible sofic shifts (in particular SFTs) as mixing sofic shifts (mixing SFTs), because they are also characterized by a topological mixing property, but we will not use that property explicitly, so we do not define it here.

The following definition is new, and we use it extensively.

\begin{defn}
Let $X$ and $Y$ be subshifts and let $\pi: X \to Y$ be a factor code. We write $R_n(\pi)$ for the set of periodic points $y \in Y$ such that $y = \pi(x)$ for some periodic point $x \in X$ with $\per(x) = \per(y)$. We write $r_n(\pi) = |R_n(\pi)|$.
\end{defn}

For a subshift $X$, the (topological) \textit{entropy} of $X$ is the value $h(X) = \inf_{n \geq 1} \frac{1}{n} \log |\cB_n(X)|$; in fact, the limit $\lim_{n \to \infty} \frac{1}{n} \log |\cB_n(X)|$ exists and is equal to $h(X)$. For a mixing sofic shift (in particular, a mixing SFT) $X$, we also have $h(X) = \lim_{n \to \infty} \frac{1}{n} \log q_n(X)$. Entropy is non-increasing under factor codes and is thus a conjugacy invariant, though certainly not a complete invariant. For any irreducible sofic shift $X$, and any proper subshift $V \subset X$, we have $h(V) < h(X)$. In \Cref{sec-count}, we use the following lemma of Marcus, which allows us to approximate a sofic shift from the inside by SFTs in terms of entropy.

\begin{lemma}[Proposition 3 in \cite{bm-1985-encod}]\label{sofic-approx-inside-sft}
Let $Y$ be a sofic shift. For every $\eps > 0$, there exists an irreducible SFT $U \subseteq Y$ with $h(U) > h(Y) - \eps$. 
\end{lemma}

For any subshift $X$ and any $k \geq 1$, we can form the \textit{$k$th higher block shift} $X^{[k]}$ with alphabet $\cB_k(X)$, where 
\[
w = (a_{1,1} a_{1,2} \dots a_{1,k}) (a_{2,1} a_{2,2} \dots a_{2,k}) \dots (a_{\ell,1} a_{\ell,2} \dots a_{\ell,k}) \in \cB(X^{[k]})
\]
if any only if for each $i,j$ we have $a_{i,j} = a_{i+1,j-1}$, so that
\[
w = (a_1 a_2 \dots a_k) (a_2 a_3 \dots a_{k+1}) \dots (a_{\ell+1} a_{\ell+2} \dots a_{\ell+k}),
\]
and $a_1 a_2 \dots a_{k+\ell} \in \cB(X)$. Observe that $X$ and $X^{[k]}$ are conjugate for any subshift $X$ and any $k \geq 1$. Moreover, if $X$ is an $m$-step SFT and $k \leq m-1$, then $X^{[k]}$ is an $(m-k)$-step SFT. In particular, every SFT is conjugate to a $1$-step SFT, and every sliding block code on an SFT can be written as a composition of a conjugacy and a $1$-block code. We will therefore frequently assume WLOG that a given sliding block code on an SFT is a $1$-block code on a $1$-step SFT.

For a sliding block code on an irreducible shift of finite type, either every fiber is a finite set (indeed, of bounded cardinality), in which case the code is said to be \textit{finite-to-one} and the entropy of the image is equal to that of the domain, or almost every fiber is uncountable, and the entropy of the image is strictly less than that of the domain. In the finite-to-one case, the minimum fiber cardinality is generic and is known as the \textit{degree}. In particular, a code (on an irreducible SFT) with degree $1$ is said to be \textit{almost invertible}. It is a theorem that every irreducible (resp. mixing) sofic shift is an almost invertible factor of an irreducible (resp. mixing) SFT. We use the following construction of almost invertible codes, known as the ``blowing-up lemma'', in the proof of \Cref{mmb-cor} in \Cref{sec-conseq}.

\begin{lemma}[Lemma 10.3.2, \cite{lm-95-intro}]\label{blowup}
Let $Z$ be a mixing SFT and let $z \in Z$ be a periodic point with least period $p$. Let $M \geq 1$. Then there exist a mixing SFT $Z'$ and an almost invertible factor code $\chi: Z' \to Z$ such that the preimage of the orbit of $z$ under $\chi$ is a single orbit of length $Mp$.
\end{lemma}

\begin{rmk}\label{rmk-blowup-details}
Note that in \cite{lm-95-intro}, the extension $\chi$ in \Cref{blowup} is only stated to be finite-to-one, but the existence of periodic points having unique preimage already implies almost invertibility. Indeed, the construction in \cite{lm-95-intro}, based on work in \cite{mb-1984-etds}, in fact shows that $\chi$ is injective except on the points that are backward-asymptotic to points in the preimage of the orbit of $z$, where we say that two points $z,z'$ are backward-asymptotic if $d(\sigma^n z, \sigma^n z') \to 0$ as $n \to - \infty$.
\end{rmk}

\subsection{Markers and Markov approximations}\label{ss-mark}

We now recall the constructions with markers and long periodic blocks that are at the heart of the proof of \Cref{main-thm}. For an alphabet $\cA$, we say that a block $w = w_1 \dots w_n \in \cA^n$ is $k$-periodic, or has self-overlap of $n-k$, if $w_{[k+1,n]} = w_{[1,n-k]}$, i.e. for $1 \leq i \leq n-k$ we have $w_i = w_{i+k}$. A given block may be $k$-periodic for several different $k$.

\begin{lemma}[Lemma 2.3 in \cite{mb-1984-etds}]\label{per-lemma}
Let $Z$ be a subshift, let $N \geq 1$, and $a, b \in \Z$ with $b - a \geq 2N$. Let $z \in Z$. If for every $i \in [a+N,b-N]$ there exists $p \leq N-1$ such that $z_{[i-N,i+N]}$ is $p$-periodic, then there is at most one periodic point $\zeta \in Z$ with $\per(\zeta) \leq N-1$ and $\zeta_{[a,b]} = z_{[a,b]}$. If $Z$ is a $1$-step SFT, then such a $\zeta$ exists. 
\end{lemma}

\begin{lemma}[Lemma 2 in \cite{wk-82-etds}]\label{marker-lemma}
Let $Z$ be a subshift. For any $N \geq 1$, there exists a subset $F \subset Z$, which can be taken to be a finite union of cylinders, such that:
\begin{enumerate}
    \item the sets $\sigma^i F$, $0 \leq i \leq N-1$, are all disjoint, and
    
    \item if $z \notin \sigma^i F$ for all $-(N-1) \leq i \leq (N-1)$, then $z_{[-N,N]}$ is $p$-periodic for some $p \leq N-1$.
\end{enumerate}
\end{lemma}

For any subshift $X$ and any $n \geq 1$, we can form the $n$th Markov approximation $X_n$, which is the SFT defined by allowing precisely the blocks of length $n$ which appear in $X$. Clearly $X_{n+1} \subset X_n$. It is an exercise to show that for any $\eps > 0$ and any $N \geq 1$, there exists $N' \geq N$ such that $h(X_{N'}) < h(X) + \eps$ and $q_n(X_{N'}) = q_n(X)$ for all $n \leq N$. In \Cref{wlog-sft}, we use the Markov approximation, together with higher block shifts, to show that in the proof of \Cref{main-thm}, we can assume WLOG that $Z$ is a $1$-step SFT, which allows us to apply \Cref{per-lemma}. 

We remark that there are versions of \Cref{marker-lemma} which obviate the need for \Cref{per-lemma}. However, for our purposes in this paper, embedding $Z$ into an SFT has the additional benefit that the rate of convergence of $\frac{1}{n} \log q_n(Z)$ to $h(Z)$ can be easily estimated when $Z$ is an SFT (see e.g. \cite{lm-95-intro}, pp. 349-351), which gives a procedure for deciding whether a given $X,Y,\pi,Z$ satisfy the periodic point condition in \Cref{main-thm}, assuming that $h(Z) < h(Y)$ (namely, compute $N \geq 1$ such that for all $n \geq N$, $q_n(Z) < r_n(\pi)$, then check all $n \leq N$ to determine whether $q_n(Z) \leq r_n(\pi)$). 

\begin{lemma}\label{wlog-sft}
Let $X$ be a mixing SFT, $Y$ a mixing sofic shift, and $\pi: X \to Y$ a factor code. Let $Z$ be a subshift with $h(Z) < h(Y)$ and $q_n(Z) \leq r_n(\pi)$ for all $n \geq 1$. Then there exists a $1$-step SFT $Z'$ such that $Z$ embeds into $Z'$, $h(Z') < h(Y)$, and $q_n(Z') \leq r_n(\pi)$ for all $n \geq 1$.
\end{lemma}

We defer the proof of \Cref{wlog-sft} to \Cref{sec-conseq}.

\section{Coding}\label{sec-code}

In \Cref{ss-blank-stamp}, we introduce two coding constructions, namely subshifts with blanks adjoined, \Cref{blanks-def}, and stamps, \Cref{stamp-def}, then use them to create one side of an interface between $Z$ on the one hand and $\pi: X \to Y$ on the other. In \Cref{ss-marker-blank}, we use markers in $Z$ to construct the other side of this interface. In \Cref{ss-stamp-sft}, we use stamps to give a construction of SFTs analogous to $S$-gap shifts. We use this construction in \Cref{ss-good-per} to construct the shifts that are used in \Cref{ss-blank-stamp} and \Cref{ss-marker-blank}.

\subsection{Blanks and stamps}\label{ss-blank-stamp}

As outlined in \Cref{intro-sec}, the proof of \Cref{main-thm} involves coding $Z$ into $X$ via certain intermediate subshifts which consist of long ``data'' blocks separated by blanks. We now define this construction precisely.

\begin{defn}[subshift with blanks adjoined]\label{blanks-def}
Let $W$ be a subshift and let $N, \ell \geq 1$ with $\ell < N$. Let $*$ be a symbol not appearing in the alphabet of $W$. Let $M \subset \bigcup_{n=1}^{2N} \cB(W)$ be a set of blocks and let $Q \subset \cup_{n=1}^{2N-1} Q_n(W)$ be a set of periodic points. Denote by $\bl(M,Q,N,*,\ell)$ the subshift in which each point is of the form $\dots w_{-1} *^{\ell} w_0 *^{\ell} w_1 \dots$ where either $w_i \in M$ or $w_i = y_T$ where $y \in Q$ and $T=(-\infty,0], [0, +\infty), (-\infty, \infty)$, or $[0,m]$ with $m \geq 2N$.
\end{defn}

The purpose of the $\bl$ construction is to provide an interface between the channel $\pi: X \to Y$ and the subshift $Z$ to be embedded. One side of this interface, namely the embedding of a $\bl$ subshift into $X$, is specified in \Cref{blanks-embed-channel}. The construction involves particular blocks, which we call \textit{stamps}, that can be unambiguously recognized in the following sense:

\begin{defn}[stamp]\label{stamp-def}
Let $Y$ be a subshift, $W \subset Y$ a proper subshift, and $k \geq 1$. We say that $\mu \in \cB(Y) \sm \cB(W)$ is a $(Y,W,k)$ \textit{stamp} if for all $u_1, u_2 \in \cB(W)$ and $v_1, v_2 \in \cB_k(Y)$, $\mu$ appears exactly once in $u_1 v_1 \mu v_2 u_2$.
\end{defn}

\begin{rmk}
In \Cref{stamp-def}, continuing with the notation there, we do not explicitly require $u_1 v_1 \mu v_2 u_2$ to be legal in $X$. Doing so would neither affect the results nor simplify the proofs. In all of the examples we consider, such blocks will in fact be legal in $X$.
\end{rmk}

\begin{prop}\label{stamp-exist}
Let $Y$ be a strongly irreducible subshift with gap $g$ and $W \subset Y$ a proper subshift. For every $k \geq g$ and every sufficiently large $n$, there exists a $(Y,W,k)$ stamp of length $n$.
\end{prop}

We defer the proof of \Cref{stamp-exist} to \Cref{ss-overlap}, but before applying stamps in \Cref{blanks-embed-channel}, we prove a lemma that expresses how stamps are actually used in our constructions.

\begin{lemma}\label{stamp-usage}
Let $Y$ be a subshift, $W \subset Y$ a proper subshift, $k \geq 1$, and $\mu \in \cB(Y) \sm \cB(W)$ a $(Y,W,k)$ stamp. Let $N \geq |\mu|$. Then for any $\gamma^{\pm} \in \cB_k(Y)$, and any $w \in \cB(W)$ with $|w| \geq N$, the stamp $\mu$ appears exactly twice in the block $\mu \gamma^- w \gamma^+ \mu$. 
\end{lemma}

\begin{proof}
By the hypotheses on $\mu$, $\gamma^{\pm}$, and $w$, and \Cref{stamp-def}, $\mu$ appears exactly once in each subblock $\mu \gamma^- w$, $w \gamma^+ \mu$. An appearance of $\mu$ other than at the positions explicitly indicated must therefore overlap both of these subblocks. Since $|w| \geq |\mu|$, $\mu$ must therefore be a subblock of $w$, contradicting the hypothesis that $w \in \cB(W)$ and $\mu \in \cB(Y) \sm \cB(W)$.
\end{proof}

We now give one of the main coding constructions (\Cref{blanks-embed-channel}), embedding a subshift with blanks adjoined, and with blocks from a subshift $V \subset X$, into $X$ via a sliding block code $\gamma$, such that $\pi \circ \gamma$ is injective. The large amount of data in the statement is representative of the complexity of the construction and the modular nature of the proof.

\begin{prop}\label{blanks-embed-channel}
Let $X$ be a mixing SFT with gap $g$, let $Y$ be a mixing sofic shift, and let $\pi: X \to Y$ be a $1$-block factor code.

Let $V \subset X$, $W = \pi(V) \subset Y$ be proper subshifts. Let $*$ be a symbol not appearing in the alphabets of $X,Y$. Let $N \geq 1$. Let $M \subset \bigcup_{n=1}^{2N-1} \cB_n(W)$ be a collection of blocks, and let $R \subset \bigcup_{n=1}^{N-1} R_n(\pi|_V)$ be a union of finite (i.e. periodic) orbits in $W$ with $\pi$-preimages of equal cardinality in $V$. Let $\kappa: M \to \cB(V)$ be an injection such that $\pi \circ \kappa(w) = w$ for each $w \in M$, and let $\Hat{M} = \kappa(M)$. Similarly, let $\lambda: R \to P(V)$ be a shift-commuting injection such that $\pi \circ \lambda(y) = y$ for each $y \in R$, and let $\Hat{R} = \lambda(R)$. Then for any $\ell \geq 1$, $\bl(M,R,N,*,\ell)$ and $\bl(\Hat{M}, \Hat{R}, N,*,\ell)$ are conjugate. 

Moreover, let $\mu \in \cB(Y) \sm \cB(W)$ be a $(Y,W,g)$ stamp such that $|\mu| \leq N$, and suppose that $M \subset \bigcup_{n=N}^{2N-1} \cB_n(W)$, i.e. $M$ contains no blocks of length less than $N$. Then there exists a sliding block code $\gamma: \bl(\Hat{M}, \Hat{R}, N,*, |\mu| + 2g) \to X$ such that $\pi \circ \gamma$ is injective.
\end{prop}

\begin{proof}
First, the conjugacy. Let $W[*] = \bl(M,R,N,*,\ell)$ and $V[*] = \bl(\Hat{M},\Hat{R},N,*,\ell)$. Consider the $1$-block code $\pi[*]$ defined on $V[*]$ by the block map $\pi[*](a) = \pi(a)$ for $a$ in the alphabet of $V$ and $\pi[*](*) = *$. We claim that $W[*] = \pi[*](V[*])$ and that $\pi[*]: V[*] \to W[*]$ is a conjugacy. To see that $W[*] \subseteq \pi[*](V[*])$, note that any $\xi \in V[*]$ is of the form
\[
\xi = \dots w_{-1} *^{\ell} w_0 *^{\ell} w_1 \dots
\]
where either $w_i \in \Hat{M}$ or $w_i = x_T$ for some $x \in \Hat{R}$ and $T$ an interval with $2N+1 \leq |T|$. If $w_i \in \Hat{M}$, then $\pi(w_i) \in M$; if $w_i = x_T$ for some $x \in \Hat{R}$, then $\pi(w_i) = \pi(x)_T$, and $\pi(x) \in R$. Therefore
\[
\pi[*](\xi) = \dots \pi(w_{-1}) *^{\ell} \pi(w_0) *^{\ell} \pi(w_1) \dots \in W[*]
\]
This shows that indeed $W[*] \subset \pi[*](V[*])$. Similarly, any $\eta \in W[*]$ is of the form
\[
\eta = \dots w_{-1} *^{\ell} w_0 *^{\ell} w_1 \dots
\]
where either $w_i \in M$ or $w_i = y_T$ for some $y \in R$ and $T$ an interval with $2N+1 \leq |T| \leq \infty$. For $\eta$ of this form, using \Cref{per-lemma}, we can use the injections $\kappa$, $\lambda$ to reconstruct a unique $\xi \in V[*]$ such that $\pi[*](\xi) = \eta$.

We now suppose that each block in $M$ has length at least $N$ and that we have a $(Y,W,g)$ stamp $\mu \in \cB(Y) \sm \cB(W)$ such that $|\mu| \leq N$. Under these assumptions, we construct a sliding block code $\gamma: V[*] \to X$ and show that $\pi \circ \gamma$ is injective. Fix a $\pi$-preimage $\Hat{\mu}$ of $\mu$, and let $\ell = |\mu| + 2g$. Using the hypothesis that $X$ is a mixing $1$-step SFT, define maps $\gamma^{\pm}: \cB_1(V) \to \cB_g(X)$ such that, for $a, b \in \cB_1(V)$, we have $\Hat{\mu} \gamma^{-}(a) a, b \gamma^+(b) \Hat{\mu} \in \cB(X)$. We then have a sliding block code $\gamma: V[*] \to X$, given by replacing each block $b *^{\ell} a$ by $b \gamma^+(b) \Hat{\mu} \gamma^-(a) a$, and leaving the non-blank symbols unchanged. 

Let
\[
\xi = \dots *^{\ell} v_{-1} *^{\ell} v_0 *^{\ell} v_1 *^{\ell} \dots \in V[*]
\]
Then
\[
\gamma(\xi) = \dots \Hat{\mu} \gamma^-(a_0) v_0 \gamma^+(b_0) \Hat{\mu} \dots
\]
where $a_i, b_i$ are, respectively, the initial and terminal symbols of $v_i$. In turn, we have
\[
\pi \circ \gamma(\xi) = \dots \mu (\pi \circ \gamma^-(a_0)) \pi(v_0) (\pi \circ \gamma^+(b_0)) \mu \dots
\]
Moreover, by \Cref{stamp-usage} and the lower bound on lengths of blocks in $M$, it follows that $\mu$ appears in $\pi \circ \gamma(\xi)$ only where $\Hat{\mu}$ appears at the same position in $\gamma(\xi)$. By replacing, in $\pi \circ \gamma(\xi)$, each appearance of $\mu$, and the blocks of length $k$ to the left and right of $\mu$, with $*^{\ell}$, we obtain the point $\dots *^{\ell} \pi(v_0) *^{\ell} \dots = \pi[*](\xi) \in \bl(M,R,N,*,\ell)$, from which $\xi$ can be recovered since $\pi[*]$ is a conjugacy.
\end{proof}

\subsection{Blanks and markers}\label{ss-marker-blank}

We now prove a lemma that encapsulates the use of marker constructions in our proof of \Cref{main-thm}.

\begin{lemma}\label{marker-black-box}
Let $Z,W$ be subshifts with $Z$ a $1$-step SFT. Let $N,\ell \geq 1$ be such that $q_n(Z) \leq q_n(W)$ for $n \leq N-1$ and $|\cB_n(Z)| \leq |\cB_{n-\ell}(W)|$ for $N+\ell \leq n \leq 2N+\ell-1$. Let $M \subset \bigcup_{n=N}^{2N-1} \cB_{n}(W)$ and $Q \subset \bigcup_{n=1}^{N-1} Q_n(W)$ be a union of finite (i.e. periodic) orbits such that $|\cB_n(Z)| \leq |M \cap \cB_{n-\ell}(W)|$ for $N+\ell \leq n \leq 2N+\ell-1$, and $q_n(Z) \leq |Q \cap Q_n(W)|$ for $n \leq N-1$. Then $Z$ embeds into $\bl(M,Q,N,*,\ell)$. 
\end{lemma}

\begin{rmk}
The lower bound on the length of blocks in $M$ is not in fact needed for \Cref{marker-black-box}, but it is needed in order to apply \Cref{marker-black-box} in conjunction with \Cref{blanks-embed-channel} in the proof of \Cref{main-thm} below. 
\end{rmk}

\begin{proof}
Let $F$ be a marker set for $Z$ with parameter $N$. For $z \in Z$, let $A(z) = \{ i \in \Z  \, | \, \sigma^i z \in F  \}$. Enumerate each $A(z)$ as $\{ a_j(z) \}_{j \in J(z)}$ where the index set $J(z)$ may be the empty set, or a finite set, or the integers, or the positive or negative natural numbers, and where $a_j(z) < a_{j+1}(z)$ for each $j$. We refer to the elements of $A(z)$ as marker coordinates for $z$. Say that $T$ is a \textit{marker interval} for $z$ if: $T = [a_j(z), a_{j+1}(z))$ where $a_j(z), a_{j+1}(z)$ are both defined; or $T=[a_0(z), \infty)$ if $a_0(z) = \max A(z) < \infty$; or $T=(-\infty, a_0(z)]$ if $a_0(z) = \min A(z) > -\infty$; or $T=(-\infty, \infty)$ if $A(z) = \emptyset$.

We construct an embedding of $Z$ into $\bl(M,Q,N,*,\ell)$ by constructing a function $\Phi$ that maps a block occurring between marker coordinates to a data block padded with $*^{\ell}$. Let $c_n: Q_n(Z) \to Q_n(W)$ be shift-commuting injections for $n \leq N-1$. Let $d_n: \cB_n(Z) \to \cB_{n-\ell}(W)$ be injections for $N+\ell \leq n \leq 2N+\ell-1$. For a block $w \in \cB_n(Z)$ with $N+\ell \leq n \leq 2N+\ell-1$, let $\Phi(w) = *^{\ell} d_n(w)$. For $z \in Z$ periodic with $n = \per(z) \leq N-1$, if $m \geq 2N+\ell$, let $\Phi(z_{[0,m]}) = *^{\ell} c_n(z)_{[\ell,m]}$. Similarly, let $\Phi(z_{[0,\infty)}) = *^{\ell} c_n(z)_{[\ell,\infty)}$. Finally, let $\Phi(z_{(-\infty,0]}) = c_n(z)_{(-\infty,0]}$ and let $\Phi(z) = c_n(z)$. Observe that $\Phi$ is injective, by \Cref{per-lemma}.

Define $\phi: Z \to W$ by declaring that $\phi(z)_T = \Phi(z_T)$ whenever $T$ is a marker interval for $z$. We need to show that $\phi$ is an embedding. Certainly $\phi$ is shift-commuting, since, if $T$ is a marker interval for $z$, then $T-1$ is a marker interval for $\sigma z$, so
\[
\phi(\sigma z)_{T-1} = \Phi((\sigma z)_{T-1}) = \Phi(z_T) = \phi(z)_T = (\sigma \phi(z))_{T-1}
\]
Thus indeed $\phi(\sigma z) = \sigma \phi(z)$. Moreover, $\phi$ is injective because the appearances of $\beta$ in $\phi(z)$ allow us to reconstruct the marker coordinates, and then the injectivity of $\Phi$ allows us to reconstruct $z_T$ for each marker interval $T$ for $z$.

We need to show finally that $\phi$ is continuous, i.e. that for $z \in Z$, $\phi(z)_0$ depends only on $z_{[-L,L]}$ for some finite $L$ independent of $z$. To see this, let $L'$ be such that $F$ is a union of cylinders on $[-L',L']$. Let $L = L'+2N$. By examining $z_{[-L,L]}$, we can determine whether there are marker coordinates for $z$ in $[-2N,0)$ and/or $[0,2N]$. If each of these intervals contains a marker coordinate, then $\phi(z)_0$ is determined by $z_T$ where $T \subset [-2N,2N]$ is the unique marker interval for $z$ containing $0$. If at least one of $[-2N,0), [0,2N]$ has no marker coordinates, then $0$ is in a long marker interval for $z$. If there is a marker coordinate in $(-\ell,0]$, then $\phi(z)_0 = *$. Otherwise, by \Cref{per-lemma}, $\phi(z)_0$ is determined by any subblock $z_{[a,b]}$ where $a <  0 \leq b$, $b-a \geq 2N$, and $[a,b]$ contains no marker coordinate for $z$. This concludes the proof that $\phi$ is continuous.
\end{proof}

The remainder of the proof of \Cref{main-thm} follows the following proposition, the proof of which is taken up in \Cref{sec-count}.

\begin{prop}\label{custom-W-exists}
Let $X$ be a mixing SFT with gap $g$, $Y$ a mixing sofic shift, and $\pi: X \to Y$ a factor code. Let $Z$ be a subshift with $h(Z) < h(Y)$ and $q_n(Z) \leq r_n(\pi)$ for every $n \geq 1$. Then there exist: $N \geq 1$, subshifts $V \subset X$, $W = \pi(V) \subset Y$, and a $(Y,W,g)$ stamp $\mu \in \cB(Y) \sm \cB(W)$, such that $|\mu| \leq N$, $q_n(Z) \leq r_n(\pi|_V)$ for $n \leq N-1$ and $|\cB_n(Z)| \leq |\cB_{n-\ell}(W)|$ for $N+\ell \leq n \leq 2N+\ell-1$, where $\ell = |\mu| + 2g$.
\end{prop}

\begin{proof}[Proof of \Cref{main-thm}]
By \Cref{wlog-sft}, assume WLOG that $Z$ is a $1$-step SFT. Let $N$, $\ell$, $V \subset X$, $W = \pi(V) \subset Y$, and $\mu$ be as in \Cref{custom-W-exists}. Let $M \subset \bigcup_{n=N}^{2N-1} \cB_{n}(W)$ be as in \Cref{marker-black-box}, and let $R \subset \bigcup_{n=1}^{N-1} R_n(\pi|_V)$ be a union of finite orbits, such that $q_n(Z) \leq |R \cap R_n(\pi|_V)|$ for $n \leq N-1$. Each of which the orbits in $R$ is, by the definition of $R_n$, necessarily the image of an orbit with equal cardinality in $V$. Here, $R$ takes the role that $Q$ plays in \Cref{marker-black-box}, but in \Cref{marker-black-box}, there was no channel $\pi$, and thus no preimage requirement, hence the change in notation. By \Cref{marker-black-box}, let $\phi: Z \to \bl(M,R,N,*,\ell)$ be an embedding. 

Let $\Hat{M}$, $\Hat{R}$ be as in \Cref{blanks-embed-channel}, let $\pi[*] : \bl(\Hat{M},\Hat{R},N,*,\ell) \to \bl(M,R,N,*,\ell)$ be a conjugacy, and let $\gamma: \bl(\Hat{M},\Hat{R},N,*,\ell) \to X$ be an embedding such that $\pi \circ \gamma$ is injective (by \Cref{blanks-embed-channel}, using $\mu$).  Then $\psi = \gamma \circ (\pi[*])^{-1} \circ \phi: Z \to X$ is a sliding block code such that $\pi \circ \psi$ is injective.
\end{proof}

\subsection{Stamps and SFTs}\label{ss-stamp-sft}

In this subsection, we prove \Cref{enlarge-mixing}, which, in conjunction with \Cref{sofic-approx-inside-sft}, allows us, in \Cref{quant-summary}, to construct a mixing SFT $V \subset X$ such that the image $\pi(V) \subset Y$ is a proper subshift of $Y$ but has entropy at least $h(Y) - \eps$ for a given $\eps > 0$. It may be possible to give a more efficient construction of such a $V$, but we have not found one. We first prove \Cref{gap-sft}, which is related to the characterization of SFTs among $S$-gap shifts (Theorem 3.3 in \cite{dad-sj-2012}).

\begin{lemma}\label{gap-sft}
Let $X$ be a mixing SFT with gap $g$ and let $V_0 \subset X$ be an SFT. Let $k \geq g$ and let $\mu \in \cB(X) \sm \cB(V_0)$ be an $(X,V_0,k)$ stamp. Let $N \geq |\mu|$ and let $V_1 \subset X$ be the closure of the set of points of the form
\[
\dots v_{-1} \gamma^+_{-1} \mu \gamma^-_0 v_0 \gamma^+_0 \mu \gamma^-_1 v_1 \dots \in X
\]
where, for each $i$, $\gamma^{\pm}_i \in \cB_k(X)$ and $v_i \in \cB(V_0)$ with $|v_i| \geq N$. Then $V_1$ is a mixing SFT.
\end{lemma}

\begin{proof}
Assume without loss of generality that $X$ is a $1$-step SFT. We first perform a small recoding for convenience, specifically to make it easier to recognize stamps, by replacing $X$ by a conjugate shift $\hat{X}$. For each $x \in X$, define $\hat{x}$ as follows: if $x_{[i,i+|\mu|)} = \mu$, then for each $i \in [-k, |\mu|+k)$, let $a = x_i$ and let $\hat{x}_i = \hat{a}$, where for symbols $a,b$ in the alphabet of $X$, we have $\hat{a} = \hat{b}$ if and only if $a=b$, and the set of symbols with hats is disjoint from the alphabet of $X$. If there is no $j \in (i-(|\mu|+k), i+k]$ with $x_{[j,j+|\mu|)} = \mu$, then let $\hat{x}_i = x_i$. Clearly the map $x \mapsto \hat{x}$ is a sliding block code, and it is just as clearly injective, since we recover $x$ from $\hat{x}$ by dropping hats. Therefore $\hat{X} = \{ \hat{x} \, | \, x \in X \}$ is a mixing SFT, conjugate to $X$. 

Denote by $\Hat{V}_1 \subset \Hat{X}$ the image of $V_1$ under the map $x \mapsto \hat{x}$. Let $\ell = |\mu| + 2k$. Since $\mu$ is an $(X,V_0,k)$ stamp, and $N \geq |\mu|$, blocks of the form $\gamma^+_i \mu \gamma^-_{i+1}$ do not overlap in any point in $V$ by \Cref{stamp-usage}, so symbols with hats occur in $\Hat{V}_1$ in blocks of length exactly $\ell$. The blocks of symbols with hats are separated by blocks from $V_0$. Since $\Hat{V}_1$ is the image of $V$ under a conjugacy $X \to \Hat{X}$, $V_1$ is an SFT if and only if $\Hat{V}_1$ is an SFT.
 
Let $m \geq N$ be such that $\hat{X}$ and $V_0$ are $m$-step SFTs. We claim that if $\Hat{x} \in \hat{X}$ is such that $x_{[i,i+m]} \in \cB_{m+1}(\Hat{V}_1)$ for all $i \in \Z$, then $\hat{x} \in \Hat{V}_1$, which means precisely that $V$ is an $m$-step SFT. To prove this claim, let $\cF \subset \cB_{m+1}(X)$ be the set of blocks of length $m+1$ which contain at least one of the following: a block of length greater than $\ell$ in which all symbols have hats; a block without hats that is not in $\cB(V_0)$; or a block of symbols without hats, of length less than $N$, bounded on both sides by symbols with hats. Note that $\cF$ is disjoint from $\cB_{m+1}(\Hat{V}_1)$. Suppose that $\hat{x}_{[i,i+m]} \notin \cF$ for all $i \in \Z$. Then any block of symbols with hats in $\hat{x}$ has length exactly $\ell$, and is thus of the form $\gamma^+ \mu \gamma^-$, where $\gamma^{\pm} \in \cB_g(X)$ (with hats added). Furthermore, the blocks separating the blocks with hats must have length at least $N$ and must be in $\cB(V_0)$, since every subblock of length $m+1$ is in $\cB(V_0)$ and $V_0$ is an $m$-step SFT. Thus indeed $\hat{x} \in \hat{V}_1$, so $\Hat{V}_1$ is indeed an SFT.

To see that $V_1$ is irreducible, let $u_-, u_+ \in \cB(V_1)$. We need to construct $u_0 \in \cB(V)$ such that $u_- u_0 u_+ \in \cB(V_1)$. We do so as follows. Extend $u_-$ on the right to form a block $v_- \in \cB(V_1)$, which begins with $u_-$ and ends with $\gamma^+_{-1} \mu \gamma^-_0$ where $\gamma^+_{-1}, \gamma^-_0 \in \cB_k(V_1)$. (It is possible that $u_-$ overlaps $\gamma^+_{-1} \mu \gamma^-_0$.) Let $v_0 \in \cB_N(V_0)$ be such that $v_- v_0 \in \cB(X)$. Similarly, extend $u_+$ on the left to form a block $v_+ \in \cB(V_1)$ which ends with $u_+$ and begins with $\gamma^+_0 \mu \gamma^-_1$, where $\gamma^+_0, \gamma^-_1 \in \cB_k(V_1)$ and $v_0 \gamma^+_0 \in \cB(X)$. Let $x^{\pm} \in \cB(V_1)$ be such that $x^-_{[0,\infty)}$ begins with $v_-$ and $x^+_{(-\infty,-1]}$ ends with $v^+$. Let $x = x^-_{(-\infty,-1]} v_- v_0 v_+ x^+_{[0,\infty)}$. Then $x \in X$ since $X$ is a $1$-step SFT. Moreover, $x \in V_1$, since the tails $x^-_{(-\infty,-1]} v_-$ and $v_+ x^+_{[0,\infty)}$ appear in $V_1$ and are joined together in a way that creates no violations of the restrictions defining $V_1$. Letting $u_0$ be the block appearing between $u_-, u_+$, such that $v_- v_0 v_+ = u_- u_0 u_+ \in \cB(V)$, the construction is complete, showing that $V_1$ is indeed irreducible.

To see that $V_1$ is mixing, let $u_1, u_2 \in \cB(V_0)$ with $|u_1| > m$, where $m$ is as above, and $|u_2| = |u_1| + 1$. Let $\gamma^{\pm}_i \in \cB(X)$, $i = 1,2$, be such that $u_i \gamma^+_i \mu \gamma^-_i u_i \in \cB(X)$. Then $x_i = (u_i \gamma^+_i \mu \gamma^-_i)^{\infty} \in V_1$ for both $i=1,2$. Indeed, certainly $x_i \in X$, since $u_i \gamma^+_i \mu \gamma^-_i u_i \in \cB(X)$ and $X$ is a $1$-step SFT. Moreover, $\per(x_i)$ divides $\ell + |u_i|$, and $\gcd(\ell + |u_1|, \ell + |u_2|) = \gcd(\ell + |u_1|, \ell + |u_1| + 1) = 1$, so $\gcd( \per(x_1), \per(x_2) ) = 1$. Since $V_1$ is an irreducible SFT with periodic points of coprime periods, $V_1$ is mixing.
\end{proof}

As advertised, we now use \Cref{gap-sft} to prove the following lemma, which is applied in the proof of \Cref{quant-summary}, which in turn is the main input to the proof of \Cref{custom-W-exists}.

\begin{lemma}\label{enlarge-mixing}
Let $X$ be a mixing SFT, $Y$ a mixing sofic shift, and $\pi: X \to Y$ a factor code. Let $W_0 \sn Y$ be an SFT. Then there exists a mixing SFT $V_1 \subset X$ with $W_0 \subset \pi(V_1) \sn Y$.
\end{lemma}

\begin{proof}
Let $V_0 = \pi^{-1}(W_0) \subset X$. Note that $V_0$ is an SFT since $W_0$ is an SFT. Let $g$ be the mixing gap of $X$. Let $y \in Y \sm W_0$ be a periodic point with least period $k \geq g$. Such a $y$ certainly exists because periodic points are dense in $Y$ and $W_0$ is a proper subshift. Let $k'$ be such that $y_{[0,k')} \notin \cB_{k'}(W_0)$. Let $\ell = k+k'$. Then every $\ell$-block in $y$ is forbidden in $W_0$. In particular, for any $x \in \pi^{-1}(\{ y \})$ and any $i \in \Z$, we have $x_{[i,i+\ell)} \notin \cB(V_0)$. 

By \Cref{stamp-exist}, let $\mu$ be an $(X,V_0,g)$ stamp. Let $V_1$ consist of the closure of the set of points of the form $\dots v_{-1} \gamma^+_{-1} \mu \gamma^-_0 v_0 \gamma^+_0 \mu \gamma^-_1 v_1 \dots \in X$ where each $v_i \in \cB(V_0)$ with $|v_i| \geq \ell$ and each $\gamma^{\pm}_i \in \cB_g(X)$. By \Cref{gap-sft}, $V_1$ is indeed a mixing SFT. Note that every point in $V_1$ contains $\ell$-blocks permitted in $V_0$, so $V_1$ is disjoint from $\pi^{-1}(\{ y \})$, and therefore $\pi(V_1) \sn Y$. 
\end{proof}

\section{Counting}\label{sec-count}

In this section, we prove \Cref{stamp-exist} and \Cref{custom-W-exists}, which state the existence and properties respectively of the stamps and the shifts $V \subset X, W \subset Y$ used in \Cref{sec-code}. \Cref{ss-overlap} contains two results required for the proof of \Cref{stamp-exist}, one (\Cref{self-overlap}) showing that most blocks in a subshift with positive entropy have little self-overlap, and the other (\Cref{synd}) showing that one can assume, at the cost of a small loss of entropy, that a given sufficiently long block appears syndetically in a mixing sofic shift. \Cref{ss-good-per} then gives a crucial asymptotic result on the number of periodic points in $Y$ with a preimage of equal least period in $X$, and applies the results from \Cref{ss-stamp-sft} to construct the shifts $V$ and $W$.

\subsection{Self-overlap and stamps}\label{ss-overlap}

We begin by showing that most blocks have very little self-overlap, which we use both to construct stamps and to determine the asymptotic number of periodic points in $Y$ with a $\pi$-preimage of equal least period.

\begin{lemma}\label{self-overlap}
Let $Y$ be a subshift with $h(Y) > 0$. For every $\alpha \in (0,1)$, there exist $N \geq 1$ and $b > 0$ such that for every $n \geq N$, there are at least $(1-\exp(-bn)) \exp(n h(Y))$ blocks $w \in \cB_n(Y)$ with no self-overlap of more than $\alpha n$. 
\end{lemma}

\begin{proof}
Let $\eps = \frac{1}{2} (\alpha^{-1} -1 )h(Y)$, so that $\alpha (h(Y) + \eps) < h(Y)$. Let $r = \exp(h(Y))$ and $s = \exp(h(Y) + \eps)$. Note that $s^{\alpha} < r < s$ and that $r^n \leq |\cB_n(Y)|$ for every $n$. Let $N_0$ be large enough that for all $n \geq N_0$, we have $|\cB_n(Y)| \leq s^n$. Let $C_1 = \sum_{k=1}^{N_0 -1} |\cB_k(Y)|$. Then the number of blocks in $X$ of length $n$ with self-overlap of more than $\alpha n$ is at most 
\begin{align*}
    \sum_{k=1}^{\lceil \alpha n \rceil } |\cB_k(Y)| &\leq \sum_{k=1}^{N_0 -1} |\cB_k(Y)| + \sum_{k=N_0}^{\lceil \alpha n \rceil } |\cB_k(Y)| \\
    &\leq C_1 + \sum_{k=N_0}^{\lceil \alpha n \rceil } s^n \\
    &\leq C_1 + \frac{ s^{\alpha n + 2} - s^{N_0} }{s-1} \\
    &\leq C_2 s^{\alpha n}
\end{align*}
where
\[
C_2 = C_1 + \frac{s^2}{s-1}
\]

Let $N > \frac{ (1-\alpha)h(Y) - \alpha \eps }{\log C_2}$. Then, for $n \geq N$, the number of blocks in $Y$ of length $n$ with no self-overlap by more than $\alpha n$ is at least 
\begin{align*}
    |\cB_n(Y)| -  \sum_{k=1}^{\lceil \alpha n \rceil } |\cB_k(Y)| &\geq r^n - C_2  s^{\alpha n} \\
    &= r^n \left( 1 - C_2 \left( \frac{s^{\alpha}}{r} \right)^n     \right) \\
    &> (1- \exp(-bn)) \exp(nh(Y))
\end{align*}
where we can take
\begin{align*}
    b &= \frac{1}{2} \log \left(  C_2 \left( \frac{r}{s^{\alpha}} \right)^N    \right) \\
    &= (1-\alpha)h(Y) - \alpha \eps - \frac{1}{N} \log C_2
\end{align*}
which is positive by the choice of $N$.
\end{proof}

We now control the entropy loss incurred by requiring a given long block to appear syndetically.

\begin{lemma}\label{synd}
Let $Y$ be a strongly irreducible subshift with $h(Y) > 0$. For every $\eps > 0$, there exist $\beta \in (0,1)$ and $N \geq 1$ such that for every $n \geq N$ and every $\theta \in \cB_{\lfloor \beta n \rfloor}(Y)$, the subshift $S \subset Y$ consisting of points $y \in Y$ in which $\theta$ appears at least once in $y_{[i,i+n)}$ for every $i \in \Z$ has entropy at least $h(Y) - \eps$.
\end{lemma}

\begin{proof}
Let $g$ be the gap for $Y$. Let $\beta = \min \{ \eps/(8 h(Y)), 1/2 \}$ and let $N = \lceil 4(2g+1)h(Y)/\eps \rceil$. Let $n \geq N$ and fix $\theta \in \cB_{\lfloor \beta n \rfloor}(Y)$. For $m \geq n$, and for all $u_1, \dots, u_k \in \cB_{(1-2\beta)n-2g}(Y)$, where $k = \lfloor m/n \rfloor$, there exist $v^{\pm}_1, \dots, v^{\pm}_k \in \cB_g(Y)$ and $v_0 \in \cB_{m-kn}(Y)$ such that
\[
v_0 \theta v^-_1 u_1 v^+_1 \theta v^-_2 u_2 v^+_2 \dots \theta v^-_k u_k v^+_k \in \cB_m(Y) 
\]
Therefore, by manipulation of logarithms and the fact that $h(Y) = \inf_{\ell \geq 1} \frac{1}{\ell} \log |\cB_{\ell}(Y) |$ by definition,
\begin{align*}
    |\cB_m(S)| &\geq |\cB_{\lfloor (1-2\beta)n\rfloor - 2g}(Y)|^{\lfloor m/n \rfloor} \\
    \frac{1}{m} \log |\cB_m(S)| &\geq \frac{1}{m} \log \left(  |\cB_{\lfloor (1-2\beta)n\rfloor - 2g}(Y)|^{(m-1)/n} \right) \\
    &= \left( 1 - \frac{1}{m}  \right) \frac{1}{n} \log |\cB_{\lfloor (1-2\beta)n\rfloor - 2g}(Y)| \\
    &\geq \left( 1 - \frac{1}{m}  \right) \frac{1}{n} ( \lfloor(1-2\beta)n\rfloor - 2g) h(Y) \\
    &> h(Y) - \eps/2
\end{align*}
for large enough $m$, where the final inequality follows from the choices of $\beta$ and $N$. We conclude that $h(S) = \liminf_{m \to \infty} \frac{1}{m} |\cB_n(S)| > h(Y) - \eps$.
\end{proof}

\begin{proof}[Proof of \Cref{stamp-exist}]
It is clearly enough to prove the result for $u_1, u_2$ sufficiently long, since we can then pass to subwords of $u_1,u_2$. By \Cref{synd}, let $\beta \in (0,1)$, $m$ sufficiently large, and $\theta \in \cB_{\beta m}(Y) \sm \cB(W)$ be such that the subshift $S \subset Y$ defined by requiring at least one appearance of $\theta$ in any block of length $m$ has $h(S) > 0$. Let $\alpha \in (0,1)$ be arbitrary, and let $n > (m+k)/(1-\alpha)$ be large enough that, by \Cref{self-overlap}, there exists $\mu \in \cB_n(S)$ such that $\mu$ has no self-overlap by more than $\alpha n$, in particular by more than $n-(m+k)$.

Let $u_1 \in \cB_{k_1}(U)$, $u_2 \in \cB_{k_2}(U)$ with $k_1, k_2 \geq m$ and let $v_1, v_2 \in \cB_{k}(Y)$. Then $\mu$ cannot appear in $u_1 v_2 \mu v_2 u_2$ except at the position explicitly indicated. Indeed, $\mu$ cannot appear at a position shifted by at most $m+k$---otherwise, $\mu$ would overlap itself by too much---and it cannot appear at a position shifted by more than $m+k$, as it would then overlap with $u_1$ or $u_2$ in a block of length at least $m$, contradicting the fact that $\mu \in \cB(S)$, and thus has $\theta$ as a subword.
\end{proof}

\subsection{Entropy and periodic points}\label{ss-good-per}

We first show that at least a positive fraction of periodic points in $Y$ of sufficient least period have a preimage of equal least period, and in particular that their growth is exponential with rate $h(Y)$.

\begin{prop}\label{growth-rn}
Let $X$ be a mixing SFT, $Y$ a mixing sofic shift, and $\pi: X \to Y$ a factor code. Then $\lim_{n \to \infty} \frac{1}{n} \log r_n(\pi) = h(Y)$. 
\end{prop}

\begin{proof}
Let $g$ be the mixing gap of $X$. By \Cref{self-overlap}, let $b > 0$ and $N > 3g$ be such that, for all $n \geq N$, the number of blocks in $Y$ of length $n-g$ with no self-overlap by more than $n/3$ is at least $c \exp(nh(Y))$, where we may take $c = \frac{1}{2} \exp(-g h(X))$. For each block $v \in \cB_{n-g}(Y)$, there exists a periodic point $x \in X$ with $\pi(x)_{[0,n-g)} = v$ such that $\per(x)$ divides $n$. Thus $\pi(x)$ is also periodic with least period dividing $n$. Moreover, if $v$ has no self-overlap by more than $n/3$, then in fact $\per(\pi(x)) = n$. Therefore $r_n(\pi) \geq c \exp(n h(Y))$, so $\liminf_{n \to \infty} \frac{1}{n} \log r_n(\pi) \geq h(Y)$, matching $\limsup_{n \to \infty} \frac{1}{n} \log r_n(\pi) \leq \lim_{n \to \infty} \frac{1}{n} \log q_n(Y) = h(Y)$.
\end{proof}

We now assemble the quantitative results proven so far.

\begin{prop}\label{quant-summary}
Let $X$ be a mixing SFT, $Y$ a mixing sofic shift, and $\pi: X \to Y$ a factor code. Let $\eps > 0$ and $N_0 \geq 1$. Then there exist $N_1 \geq N_0$ and proper subshifts $W \sn Y$, $V = \pi^{-1}(W) \subset X$, such that: $h(W) > h(Y) - \eps$; for $n \leq N_1$, $r_n(\pi|_V) = r_n(\pi)$; and for $n \geq N_1$, $r_n(\pi|_V) > \exp(n (h(Y) - \eps))$.
\end{prop}

\begin{proof}
By \Cref{sofic-approx-inside-sft} and \Cref{enlarge-mixing}, let $V_1 \subset X$ be a mixing SFT such that $h(Y) - \eps/2 < h(\pi(V_0)) < h(Y)$. Let $W_1 = \pi(V_1)$. By \Cref{growth-rn}, let $N_1 \geq N_0$ be such that for any $n \geq N_1$, we have $\frac{1}{n} \log r_n(\pi|_{V_1}) > h(W_1) - \eps/2 > h(Y) - \eps$. Let $W = W_1 \cup \bigcup_{n=1}^{N_1} R_n(\pi)$ and $V = \pi^{-1}(W)$. Then $r_n(\pi|_V) = r_n(\pi)$ for all $n \leq N_1$.

To see that $W \neq Y$, observe that the only $n$-blocks in $W$ that may not be in $W_1$ are those in the low-order periodic points that have been adjoined, which are bounded in number by a constant. That is, $|\cB_n(W)| \leq |\cB_n(W_1)| + C$ for all $n \geq N_1$, where we can take $C = \sum_{k=1}^{N_1} k |R_k(\pi)|$. Thus $h(W) = h(W_1) < h(Y)$.
\end{proof}

\Cref{quant-summary} is the final input to the proof of \Cref{custom-W-exists}, and thus of \Cref{main-thm}.

\begin{proof}[Proof of \Cref{custom-W-exists}]
Let $\eps = h(Y) - h(Z)$. Let $N_0 \geq 1$ be large enough that for all $n \geq N_0$,
\[
\frac{1}{n} \log \max \{ q_n(Z), |\cB_n(Z)|  \} < h(Z)+ \frac{\eps}{4}
\]
By \Cref{quant-summary}, let $W \subset Y$, $V = \pi^{-1}(W) \subset X$, and $N_1 \geq N_0$ be such that $h(W) > h(Y) - \frac{\eps}{4}$, $r_n(\pi|_V) = r_n(\pi)$ for all $n \leq N_1$, and $\frac{1}{n} \log  r_n(\pi|_V) > h(Y) - \frac{\eps}{4}$ for all $n \geq N_1$. Note that $h(W) > h(Z) + \frac{\eps}{2}$ and that $q_n(Z) \leq r_n(\pi|_V)$ for all $n \geq 1$.

Let $g$ be the mixing gap of $X$. By \Cref{stamp-exist}, let $\mu \in \cB(Y) \sm \cB(W)$ be a $(Y,W,g)$ stamp. Let $\ell = |\mu| + 2g$. Then since $h(Z) < h(W)$, there exists $N$ sufficiently large so that for all $n \geq N$, in particular for $N+\ell \leq n \leq 2N+\ell-1$, we have $|\cB_n(Z)| < |\cB_{n-\ell}(W)|$.
\end{proof}

\section{Proofs of \Cref{wlog-sft} and \Cref{mmb-cor}}\label{sec-conseq}

We first use \Cref{growth-rn}, along with facts about Markov approximations in \Cref{ss-mark}, to prove \Cref{wlog-sft}, which reduces \Cref{main-thm} to the case where $Z$ is a $1$-step SFT.

\begin{proof}[Proof of \Cref{wlog-sft}]
We use the properties of Markov approximations mentioned in \Cref{conv-sec}. Let $\eps = h(Y)-h(Z)$. Let $m_0$ be such that $h(Z_{m_0}) < h(Z) + \eps/3$, where $Z_{m_0}$ is the $m_0$th Markov approximation to $Z$, and such that, by \Cref{growth-rn}, for all $n \geq m_0$, $\frac{1}{n} \log r_n(\pi) > h(Y) - \eps/3$. Let $m_1 \geq m_0$ be such that $\frac{1}{n} \log q_n(Z_{m_0}) < h(Z_{m_0}) + \eps/3$ for all $n \geq m_1$. Let $m_2 \geq m_1$ be such that for all periodic points $z \in P(Z_{m_1}) \sm Z$ (under the natural embedding $Z \hookrightarrow Z_{m_1}$) with $\per(z) \leq m_1$, we have $z_{[0,m_2)} \notin \cB_{m_2}(Z)$. Then $Z_{m_2}$ satisfies $q_n(Z_{m_2}) = q_n(Z) \leq r_n(\pi)$ for all $n \leq m_1$. Moreover, since $Z_{m_2} \subset Z_{m_0}$, $\frac{1}{n} \log q_n(Z_{m_2}) \leq \frac{1}{n} \log q_n(Z_{m_0})$ for all $n$; in particular, 
\begin{align*}
\frac{1}{n} \log q_n(Z_{m_2}) &< h(Z_{m_0}) + \eps/3 \\
&< h(Z) + 2\eps/3 \\
&< h(Y) - \eps/3 \\
&< \frac{1}{n} \log r_n(\pi)
\end{align*}
for all $n \geq m_1$. Taking $Z' = Z_{m_2}^{[m_2]}$ to be the $m_2$th higher block shift, the lemma is proved.
\end{proof}

To prove \Cref{mmb-cor}, in the mixing sofic case, we use \Cref{blowup} to handle low-order periodic point obstructions, with periodic points of sufficiently high order controlled by \Cref{growth-rn}. To handle the arbitrary case, we first give an improved Markov approximation (\Cref{embed-mixing}), embedding an arbitrary subshift into a mixing SFT with only slightly greater entropy. The construction uses \Cref{gap-sft}; in \Cref{gap-sft-entropy} we estimate the entropy of the mixing SFT constructed in \Cref{gap-sft}.

\begin{lemma}\label{gap-sft-entropy}
Let $X$ be a mixing SFT with gap $g$ and let $V_0 \subset X$ be an SFT. Let $k \geq g$ and let $\mu \in \cB(X) \sm \cB(V_0)$ be an $(X,V_0,k)$ stamp. For any $\eps > 0$, there exists $N \geq |\mu|$ such that $h(V_1) < h(V_0) + \eps$, where $V_1$ (depending on $N$) is as in \Cref{gap-sft}.
\end{lemma}

\begin{proof}
Let $N_0 \geq 1$ be such that for all $n \geq N_0$ we have $\frac{1}{n} \log |\cB_n(V_0)| < h(V_0) + \eps/4$. Let $N > 2N_0$ be such that 
\[
\frac{1}{N} \max \{ \log N,  \log |\cB_k(X)|^2, \log |\cB_{N_0}(V_0)|\} < \frac{\eps}{4}
\]

We will show that $\frac{1}{N} \log |\cB_N(V_1)| < h(V_0) + \eps$. Consider a block of length $N$ in $V_1$. Such a block can contain at most one full or partial block of the form $\gamma^+ \mu \gamma^-$ where $\gamma^{\pm} \in \cB_k(X)$. The $\mu$, if present, can begin at any of the $N$ positions. The rest of the block of length $N$, outside the block $\gamma^+ \mu \gamma^-$, consists of one or two blocks from $V_0$, with length totalling at most $N$. We thus have 
\[
|\cB_N(V_1)| \leq N |\cB_k(X)|^2 \max_{0 \leq \ell \leq N/2} |\cB_{\ell}(V_0)| |\cB_{N-\ell}(V_0)| 
\]
If $0 \leq \ell \leq N_0$, then $|\cB_{\ell}(V_0)| |\cB_{N-\ell}(V_0)|  \leq |\cB_{N_0}(V_0)| |\cB_{N}(V_0)|$, so 
\begin{align*}
    \frac{1}{N} \log \left( |\cB_{\ell}(V_0)| |\cB_{N-\ell}(V_0)| \right) &\leq \frac{1}{N} \log |\cB_{N_0}(V_0)| + \frac{1}{N} \log |\cB_{N}(V_0)| \\
    &< h(V_0) + \frac{\eps}{2}
\end{align*}
If $N_0 \leq \ell \leq N/2$, then
\begin{align*}
    \frac{1}{N} \log \left( |\cB_{\ell}(V_0)| |\cB_{N-\ell}(V_0)| \right) &< \frac{1}{N}  \left( \ell \left(h(V_0) + \frac{\eps}{4} \right)  +  (N-\ell) \left(h(V_0) + \frac{\eps}{4} \right)  \right) \\
    &= h(V_0) + \frac{\eps}{4}
\end{align*}
Therefore 
\begin{align*}
    \frac{1}{N} \log |\cB_N(V_1)| &< \frac{1}{N} \log N + \frac{1}{N} \log |\cB_k(X)|^2 + h(V_0) + \frac{\eps}{2} \\
    &< h(V_0) + \eps
\end{align*}
by the above choice of $N$.
\end{proof}

along with the following standard lemma improving the construction of the Markov approximation. For completeness, we include a proof using \Cref{gap-sft}.

\begin{lemma}\label{embed-mixing}
Let $Z$ be a subshift and let $\eps > 0$. Then there exists a mixing SFT $V$ containing $Z$ with $h(V) < h(Z) + \eps$.
\end{lemma}

\begin{proof}
Let $m$ be large enough that $h(Z_m) < h(Z) + \eps/2$, where $Z_m$ is the $m$th Markov approximation to $Z$. Let $X$ be the full shift on the alphabet of $Z$. Certainly $Z_m \subseteq X$. If $Z_m = X$, then we can take $V=X$. If $Z_m \neq X$, then by \Cref{stamp-exist}, let $\mu$ be an $(X,Z_m,k)$ stamp for some $k \geq 0$. Let $V_0 = Z_m$ and $V=V_1$ as in \Cref{gap-sft} where $N$ is large enough that, by \Cref{gap-sft-entropy}, we have $h(V) < \eps/2$. Thus $V$ is indeed a mixing SFT containing $Z$ with $h(V) < h(Z) + \eps$.
\end{proof}

\begin{proof}[Proof of \Cref{mmb-cor}] We first consider the case in which $Z$ is mixing sofic. Let $\Tilde{Z}$ be a mixing SFT and $\chi_0: \Tilde{Z} \to Z$ an almost invertible factor code. If already $q_n(\Tilde{Z}) \leq r_n(\pi)$ for all $n \geq 1$, then we can take $Z' = \Tilde{Z}$ and apply \Cref{main-thm} immediately to construct the claimed embedding $\psi: Z' \to X$. However, if $q_n(\Tilde{Z}) > r_n(\pi)$ for some $n$, so that $X,Y,\pi, \Tilde{Z}$ violate the hypotheses of \Cref{main-thm}, then we need to construct a further extension of $\Tilde{Z}$ which satisfies the hypotheses of \Cref{main-thm}. The construction, consisting a tower of extensions via \Cref{blowup}, is as follows. 

By \Cref{growth-rn}, since $h(\Tilde{Z}) = h(Z) < h(Y)$, there are at most finitely many $n$ such that $q_n(\Tilde{Z}) > r_n(\pi)$. Let $N$ denote the greatest such $n$. Let $C = \sum_{k=1}^{N} \max\{0, q_k(\Tilde{Z}) - r_k(\pi)\}$. That is, $C$ is the number of periodic points by which $X,Y,\pi, \Tilde{Z}$ violate the hypotheses of \Cref{main-thm}. For $1 \leq k \leq N$ and $1 \leq \ell \leq k^{-1} \max\{0, q_k(\Tilde{Z}) - r_k(\pi)\}$, let $z_{k,\ell}$ be periodic points with pairwise disjoint orbits, such that $\per(z_{k,\ell}) = k$. For a given $k$, the union of the orbits of the points $z_{k,\ell}$ has cardinality $\max\{0, q_k(\Tilde{Z}) - r_k(\pi)\}$. Let $C' = \sum_{k=1}^{N} k^{-1} \max\{0, q_k(\Tilde{Z}) - r_k(\pi)\}$ (counting orbits, rather than points), and let $\{ z^{(j)}  \}_{j=1}^{C'} = \{ z_{k,\ell} \}_{k,\ell}$ be an enumeration of the points $z_{k,\ell}$.

Again by \Cref{growth-rn}, let $M > N$ be large enough that for all $n \geq M$, we have $q_n(\Tilde{Z}) + Cn \leq r_n(\pi)$. We now repeatedly apply \Cref{blowup}. Let $Z^{(0)} = \Tilde{Z}$. For $1 \leq j \leq C'$, let $Z^{(j)}$ be a mixing SFT and $\chi^{(j)}: Z^{(j)} \to Z^{(j-1)}$ an almost invertible factor code such that the preimage of the orbit of $z_j$ under $\chi^{(j)}$ is a single orbit of length $M \per(z_j)$, and such that every periodic point in $Z^{(j-1)}$ not in the orbit of $z_j$ has a unique preimage under $\chi^{(j)}$. Let $\eta^{(1)} = \chi^{(1)}$ and $\eta^{(j+1)} = \eta^{(j)} \circ \chi^{(j+1)}$. Let $Z' = Z^{(C')}$ and $\eta = \eta^{(C')} : Z' \to \Tilde{Z}$. Certainly $\eta$ is almost invertible, so $h(Z') = h(\Tilde{Z}) < h(Y)$. We claim that $q_n(Z') \leq r_n(\pi)$ for all $n \geq 1$. Indeed, for each $j$, if $\per(z_j) = k$, then we have $q_k(Z^{(j)}) = q_k(Z^{(j-1)}) - k$, $q_{Mk}(Z^{(j)}) = q_{Mk}(Z^{(j-1)}) + Mk$, and $q_n(Z^{(j)}) = q_n(Z^{(j-1)})$ for all $n \notin \{  k, Mk \}$. Therefore $q_k( Z' ) = r_k(\pi)$, and
\begin{align*}
    q_{Mk}(Z') &= q_{Mk}(\Tilde{Z}) + M \max\{0, q_k(\Tilde{Z}) - r_k(\pi)\} \\
    &\leq q_{Mk}(\Tilde{Z}) + CM \\
    &\leq r_{Mk}(\pi)
\end{align*}
where the last inequality follows from the choice of $M$. Therefore $X,Y,\pi,Z'$ satisfy the hypotheses of \Cref{main-thm}, so there exists a sliding block code $\psi: Z' \to X$ such that $\pi \circ \psi$ is injective. This concludes the proof in the case that $Z$ is mixing sofic.

We now handle the general case, where $Z$ is an arbitrary subshift with $h(Z) < h(Y)$. By \Cref{embed-mixing}, let $V$ be a mixing SFT containing $Z$ with $h(V) < h(Y)$. By the mixing sofic case, let $V'$ be a mixing SFT such that $X,Y,\pi,V'$ satisfy the hypotheses of \Cref{main-thm}, and let $\chi: V' \to V$ be an almost invertible factor code. Let $Z' = \chi^{-1}(Z)$. Then $\chi|_{Z'}$ is still finite-to-one, which concludes the proof.
\end{proof}

\section*{Acknowledgments}

I thank Tom Meyerovitch for posing the question that led to this work and suggesting the argument for \Cref{growth-rn}, and more generally for offering, with Brian Marcus, very generous advice and supervision. I also thank Mike Boyle for a helpful conversation by email about the relationship between my question \cite{vs-it-22-mo} and the embedding problem for sofic shifts.

\bibliographystyle{plain}

\end{document}